\documentclass{article}
\usepackage{graphicx} % Required for inserting images
\usepackage[a4paper, left=3cm, right=3cm, top=2cm]{geometry}
\usepackage{amsmath}
\usepackage{blindtext}
\usepackage{amssymb}
\usepackage[english]{babel}
\usepackage{amsthm}
\usepackage{faktor}
\usepackage{xcolor}
\usepackage{mathtools}
\usepackage{tikz-cd}
\usepackage{etoolbox}
\usepackage{enumitem}
\usepackage{scrextend}
\usepackage[parfill]{parskip}
\usepackage{titlesec}
\usepackage{csquotes}
\usepackage{abstract}
\usepackage{xfakebold}
\usepackage{faktor}
\usepackage{comment}
\usepackage{breqn}
\usepackage{graphicx}
\graphicspath{ {./images/} }

\usepackage[
colorlinks=true,
allcolors = {blue}
]{hyperref}

\theoremstyle{definition}
\newtheorem{theorem}{Theorem}

\theoremstyle{definition}

\newtheorem{lemma}{Lemma}
\newtheorem{prop}[lemma]{Proposition}

\newtheorem*{theorem*}{Theorem}

\newtheorem*{remark}{Remark}
\let\emptyset\varnothing
\let\originaliota\iota
\renewcommand{\iota}{\dot\originaliota}

\AtBeginEnvironment{theorem}{\setlist{before=\leavevmode, nosep}}

\titleformat{\section}
  {\Large\color{black}}
  {\thesection}{1em}{}
\titleformat{\subsection}
  {\Large\color{black}}
  {\thesubsection}{1em}{}
  \titleformat{\subsubsection}
  {\color{black}}
  {\thesubsubsection}{1em}{}
\titleformat{\theb}
  {\Large\color{black}}
  {\thesection}{1em}{}

\newcommand{\monthyear}{\ifcase \month \or January\or February\or March\or %
April\or May \or June\or July\or August\or September\or October\or November\or 
December\fi, \number \year}

\title{Differentials on Forested and Hairy Graph Complexes with Dishonest Hairs}
\author{Nicolas Grunder \thanks{I am very grateful for any potential corrections, remarks, comments, and additions. \\E-mail address: nigrunder@student.ethz.ch.}}

\date{}
\begin{document}

\maketitle

\begin{center}\large Abstract \end{center} We study the cohomology of forested graph complexes with ordered and unordered hairs whose cohomology computes the cohomology of a family of groups $\Gamma_{g,r}$ that generalize the (outer) automorphism group of free groups. We give examples and a recipe for constructing additional differentials on these complexes. These differentials can be used to construct spectral sequences that start with the cohomology of the standard complexes. We focus on one such sequence that relates cohomology classes of graphs with different numbers of hairs and compute its limit.

\section{Introduction}
Graph complexes are combinatorial algebraic objects that can be connected to different fields of mathematics. These complexes come in many different flavors, defined in specific cases as it seems appropriate. In this paper, we mainly study two variants of graph complexes. In hairy graph complexes $HG_n$ for $n=0,1$ graphs have ordered or honest hairs and unordered or dishonest hairs. In forested graph complexes $FG_n$ for $n=0,1$ graphs have additionally the data of a marked subset of edges that form a forest. The definitions of these complexes are given in section \ref{section complexes}. 
\begin{figure}[h]
    \centering
    \includegraphics[scale = 0.35]{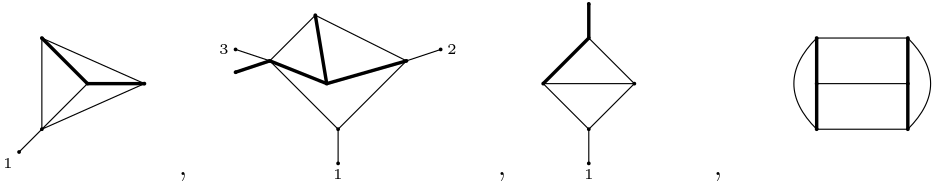}
    \caption{Examples of admissible graphs in $FG_n$.}
    \label{fig:Graph examples}
\end{figure}

In most literature graphs in graph complexes have either unordered or ordered hairs, in this paper, we slightly enlarged the complexes such that graphs are allowed to have both hair types \cite{DGC2,BrunWillwacher}. In \cite{DGC,DGC2} undecorated and hairy graph complexes have been studied in a very similar fashion as we will do here. The main new contribution in this paper is a result for the forested graph complexes. The cohomology of subcomplexes $FG_n^{g-loop}(r,0)$ of graphs of loop order $g$ with $r$ ordered hairs and without unordered hairs compute the cohomology of a family of groups $\Gamma_{g,r}$ (see \cite{HatcherVogtmann}) that generalize the outer automorphism group of free groups $\Gamma_{g,0} = Out(F_g)$ and the automorphism groups of free groups $\Gamma_{g,1} = Aut(F_g)$ for $2g+r\geq 3$ (see \ref{Prop AutoGroups} for precise statement), cf. \cite{BrunWillwacher,ContantKassabovVogtmann,ContantVogtmann}.

We will introduce a very simple differential $\delta_h$ on forested and hairy graph complexes. The map $\delta_h$ attaches an unordered hair to vertices. As it is, the sum of this differential with the standard differential $\delta+\delta_h$ only defines a differential on $HG_n$ but not on $FG_n$. In section \ref{section differentials} we extend $\delta_h$ to $$D= \exp(\Lambda)\delta_h\exp(-\Lambda) = \delta_h+\delta_{h,e}$$
such that $(\delta+D)^2 = 0$ on $FG_n$. The differential $\delta_{h,e}$ attaches unordered hairs to unmarked edges (Fig. \ref{fig: diff h,e}).
\begin{figure}[h!]
    \centering
    \includegraphics[scale = 0.2]{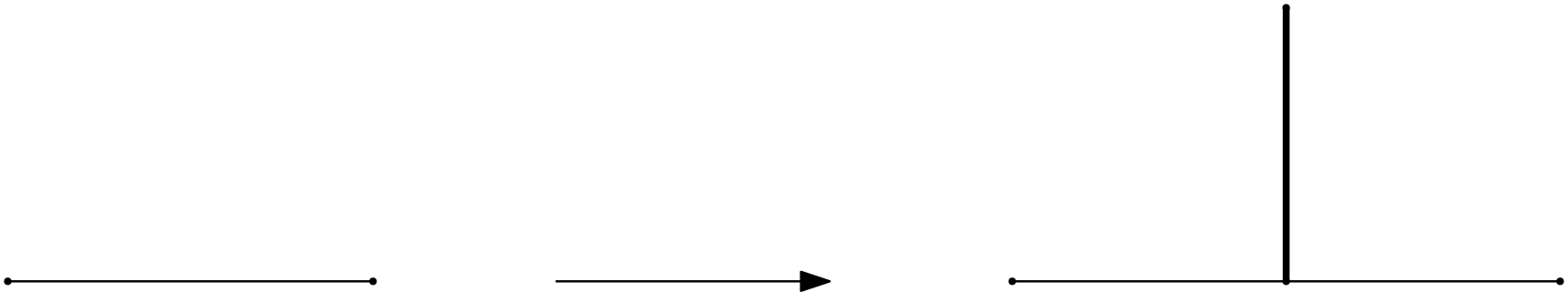}
    \caption{Action of $\delta_{h,e}$.}
    \label{fig: diff h,e}
\end{figure}

We obtain the following results by filtering the complex with differential $\delta+D$ by the number of hairs.
\begin{theorem}\label{Thm FG}
    For each $n= 0,1$ there is a spectral sequence with term $$E_1 = H(FG_n,\delta_s+\delta_m)$$
       that converges to the graded vector space $$H_n = \prod_{\substack{ k\geq 1\\ k \equiv (-1)^n \text{ mod }4}}HH^k.$$
\end{theorem}
The $HH^k$ denotes one-dimensional vector spaces corresponding to hedgehog graphs (Figure \ref{fig:FullHedgehog}) formed by a single loop with $\frac{k+1}{2}$ internal edges and dishonest hairs. The following result is not new but is proved analogously as \ref{Thm FG}.
\begin{theorem}[\cite{DGC2,TurchinWillwacher1,TurchinWillwacher2,Willwacher}]\label{Thm HG}
       For each $n= 0,1$ there is a spectral sequence with term $$E_1 = H(HG_n,\delta_s)$$
       that converges to $$H_n[-1] = \prod_{\substack{ k\geq 2\\ k-1 \equiv (-1)^n \text{ mod }4}}HH^k$$ the twisted complex by degree -1.
\end{theorem}
\begin{figure}[h]
    \centering
    \includegraphics[scale = 0.5]{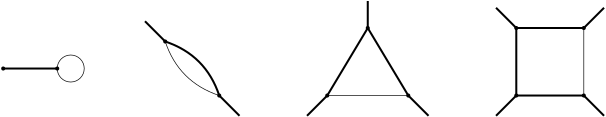}
    \caption{Graphs corresponding to $HH^k$.}
    \label{fig:FullHedgehog}
\end{figure}
These spectral sequences show that most of the homology classes come in pairs, represented by graphs with a different number of unordered hairs.
As will be clear after the definition of the graph complexes $FG_n$ and $HG_n$ in section \ref{section complexes}, there are subcomplexes $FG_n(r,k)$ and $HG_n(r,k)$ with $r$ honest hairs and $k$ dishonest hairs, let $C(r,k)$ denote either one of them. Then $HG_n$ or $FG_n$ given by a product of subcomplexes $C(r,k)$ for $r,k\in \mathbb Z_{\geq 0}$. The spectral sequences in Theorem \ref{Thm FG} and Theorem \ref{Thm HG} relate the homologies of $H(C(r,k))$ for different $k\geq 0$. To any homology class in $H(C(r,k))$ that is not a hedgehog graph (defined in section \ref{section spec seq}) a different homology class in $H(C(r,l))$ can be assigned to form pairs that cancel in the spectral sequences. The hedgehog graphs survive to $E_\infty$ and correspond to elements in $H_n$ or $H_n[1]$. The spectral sequence shows how one can construct new nontrivial homology classes out of known nontrivial homology classes. 

\subsection{Structure of the paper} In section \ref{section complexes} we define the complexes $HG_n$ and $FG_n$ by first defining a larger complex $MG_n$ whereof they will descend. In section \ref{section differentials} we introduce three additional differentials on $FG_n$ and $HG_n$. Then we will concentrate on one of them in section \ref{section spec seq} and use it to construct the spectral sequence from Theorem \ref{Thm FG} and \ref{Thm HG}.
\section{Definition of $MG_n,$ $HG_n,$ and $FG_n$} 
\label{section complexes}
We mainly analyze two types of graph complexes. Complexes of simpler type are called hairy graph complexes and are a vector space spanned by series of graphs $G=(V,E)$ where any vertex is either univalent or of valence greater than two. Edges connected to univalent vertices are called hairs and come in two types. A subset of these hairs, called honest hairs, are ordered while the other hairs, called dishonest hairs, are not ordered. Two hairy graphs are equivalent if there is a graph isomorphism that respects the order of honest hairs. The graph complexes come in two flavors, we denote them by $HG_n$ for $n = 0,1$. For $n=0$ edges are oriented and for $n=1$ halfedges and vertices are oriented. Permuting oriented objects of a graph is identified with an additional sign given by the permutation. We identify two graphs with opposite orientation by an additional sign. We grade this vector space by the number of edges of graphs. To form a complex, we define a map $\delta_s$ of degree 1 that squares to zero. On graphs, the map $\delta_s$ acts by summing over all possible ways of splitting a vertex into two vertices of valence greater than two connected by an edge (Fig. \ref{im: diff s}). 
\begin{figure}[h]
    \centering
    \includegraphics[scale = 0.4]{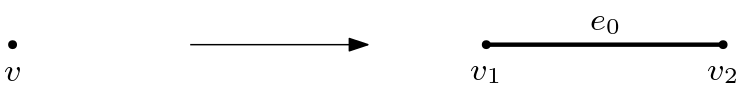}
    \caption{The action of $\delta_s$ on vertices.}
    \label{im: diff s}
\end{figure}

The second type of complexes are called complexes of forested graphs and are closely related to the complexes of hairy graphs. There are again two types of these complexes, denoted by $FG_n$ for $n=0,1$. They are spanned by series of graphs that only have univalent vertices or vertices of valence greater than two. Similarly, there are ordered honest hairs and unordered dishonest hairs. We endow these graphs with additional data of marked edges that form a forest $M\subset E$ called marking. The marking must be such that honest hairs are not marked and dishonest hairs are marked. For $n=0$ we orient marked edges while for $n=1$ we orient unmarked edges, half-edges, and vertices. Permuting oriented objects is identified with an additional sign of the permutation.  We grade this vector space by the number of marked edges. The differential on forested graph complexes is defined by $\delta_s +\delta_m$ where $\delta_s$ acts on graphs by splitting vertices into a marked edge and $\delta_m$ sums over all possible ways of marking an edge (Figure \ref{fig:diff m}).  \\

\begin{figure}[h]
    \centering
    \includegraphics[scale = 0.6]{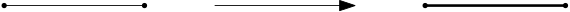}
    \caption{The action of $\delta_m$.}
    \label{fig:diff m}
\end{figure}

To define the complexes more formally we will first define a large complex $MG_n$ of which the complexes $HG_n$ and $FG_n$ will descend naturally. \\
We say a graph $G = (V,E)$ is \emph{admissible} if any vertex is univalent or has valence greater than two. A priori, tadpoles and multiple edges are allowed. Univalent vertices are called \emph{leaves} and edges attached to leaves are called \emph{hairs}. Vertices that are not leaves are called \emph{internal}, the set of internal vertices is denoted by $V^{int}\subset V$. Edges that are not hairs are called interior edges. The interior edges are denoted by $E^{int} \subset E$. Any graph is given the data of a subset of hairs denoted by $HR$ called \emph{honest hairs} and an ordering of honest hairs $\xi:|HR|\to HR$. Hairs that are not honest are called \emph{dishonest hairs}. The set of dishonest hairs is denoted by $DHR$. A \emph{marking} is a subset $M\subset E$ such that $DHR\subset M$ and $HR\cap M = \emptyset$, i.e. dishonest hairs must be marked and honest hairs must not be marked. An \emph{admissible triple} is given by $(G,M,\xi)$ where $G$ is a graph, $M$ is the marking, and $\xi: |HR| \to HR$ is an ordering of the honest hairs.\\
The definitions of the complexes $HG_n$ and $FG_n$ depend on an integer $n = 0,1$. We define two notions of orientations of an admissible triple $(G,M,\xi)$ corresponding to $n = 0,1$. For $n = 0$ an orientation of a triple $(G,M,\xi)$ is an element $$or(G,M) = or(G) = x_1\wedge...\wedge x_{|M|}\in \bigwedge\mathbb Q^{M}.$$
That is, marked edges are oriented.

For $n=1$ the definition of an orientation is more complicated. A \emph{half edge} in a graph is a pair $(v,e) \in V \times E$ such that $e$ is attached to $v$. Let us denote the half edges by $HE$. Then an orientation of an admissible pair is an element $$ or(G,M) = or(G) = x_1\wedge ...\wedge x_{|V\cup E-M\cup HE|} \in \bigwedge \mathbb Q^{V\cup E-M\cup HE}.$$
That is vertices, unmarked edges, and half edges are oriented. \\
Two admissible triples $(G,M,\xi)$ and $(G',M',\xi')$ are isomorphic if there is a graph isomorphism $\Phi: G\to G'$ that induces a bijection of the marking $M \to M'$ and preserves the ordering of honest hairs $\Phi\circ \xi' = \xi$. Admissible triples are orientation isomorphic if the isomorphism is orientation-preserving. Let $MG_n^{k,l}$ denote the vector space spanned over $\mathbb Q$ by orientation isomorphism classes of admissible triples with $k$ vertices and $l$ marked edges, where additionally we identify $\lambda(G,or(G)) = (G,\lambda or(G))$ for $\lambda\in \mathbb Q$. By abuse of notation, we write $G$ for the triple $(G,M,\xi)$ and $(G,or(G))$ for the orientation isomorphism classes. In the rest of the paper, not admissible graphs are identified with zero. The \emph{space of marked graphs} is assembled in the following way $$MG_n = \prod_{k,l}MG_n^{k,l}.$$
This space, with elements that are series of graphs, is graded by the number of marked edges. For the sake of simplicity, we will always refer to elements of $MG_n$ with the least amount of information necessary in the specific case. For example, if the marking is not important at the moment, we might write $(G,or(G))$ rather than $(G,M,\xi,or(G))$.  \\ \\
In the following, we define differentials $\delta$ on $MG_n$ that are given by the sum of two maps $\delta = \delta_s+\delta_m$. The simpler map of the two is $\delta_m$. 
On graphs with orientation $((G,M),or(G)) \in MG_n$ we define $$\delta_m((G,M),or(G)) = \sum_{e\in E-(M\cup HR)} ((G,M\cup\{e\}),\delta_m^e or(G)).$$
The definition of $\delta_m^e or(G)$ differs for $n=0$ and $n=1$. Namely, for $n = 0$ we simply say $$\delta_m^eor(G) = e\wedge or(G).$$
For $n = 1$ we have that $e$ is in the orientation $or(G)$. When $or(G) \neq 0$ we set $$d_e or(G) = d_e( e\wedge x_1 \wedge ... \wedge x_l) =  x_1\wedge...\wedge x_l$$
leaving out $e$ where $l = {|V\cup E-(M\cup \{e\}) \cup HE|}$. Further, we set $d_e 0 = 0$. Note that $$d_{e_1}d_{e_2} or(G) = - d_{e_2}d_{e_1} or(G).$$ We set $$\delta_m^eor(G) = d_e or(G).$$
Before defining $\delta_s$, we define a map $\Delta$ of degree $0$. This map corresponds to splitting a vertex in an unmarked edge. Let $G$ be an admissible graph with marked edges, $v$ a vertex of $G$, and $H$ be a graph with marked edges. By \emph{graphs obtained by inserting $H$ into $G$ at a vertex $v$ of $G$} we mean graphs that can be constructed as follows. First consider $G\cup H$ and remove $v$ from $G$. Now reconnect edges once connected to $v$ to some vertex in $H$ such that the resulting graph is admissible and such that no vertex corresponding to a vertex of $H$ is univalent. \\
 Let $H$ be the graph given by two vertices $v_1,v_2$ and an unmarked edge $e_0$ connecting them. Let $G$ be an admissible graph with marked edges and let $v\in V$. Set $S_v(G)$ to be the set graphs that can be obtained by inserting $H$ into $v$. Then we let $$\Delta(G,or(G)) = \sum_{v\in V} \sum_{ H\in S_v(G)}(H, \Delta^vor(G)).$$
The orientation $\Delta^vor(G)$ differs for $n = 0 $ and $n=1$. For $n = 0$ the orientation remains unchanged $$\Delta^vor(G) =  or(G).$$ For $n = 1$ we set $$\Delta^vor(G) = v_1\wedge (v_1,e_0)\wedge e_0\wedge (v_2,e_0)\wedge v_2 \wedge d_vor(G).$$
The differential $\delta_s$ is now given by the commutator of $\Delta$ and $\delta_m$. That is $$\delta_s = [\delta_m,\Delta] = \delta_m\Delta-\Delta\delta_m.
$$
\begin{figure}[h]
    \centering
    \includegraphics[scale =0.4]{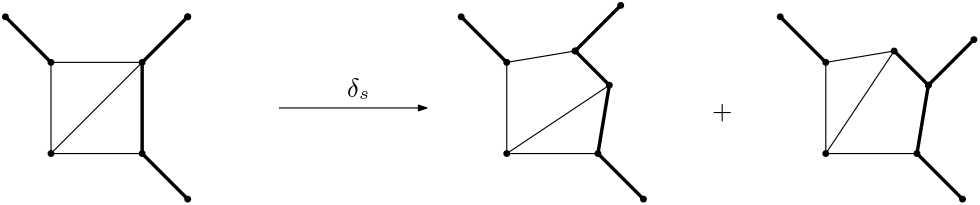}
    \caption{Example of $\delta_s$.}
    \label{fig:diff s example}
\end{figure}
\begin{remark}
    An alternative but equivalent definition of $\delta_s$ is the following. Let $H$ be the same graph as before but with a marked edge $e_0$. Let $G$ be a graph and let $S_v(G)$ denote the set of graphs obtained by inserting $H$ into $G$ at $v$. Define $$\delta_s(G,or(G)) = \sum_{v\in V} \sum_{H \in S_v(G)}(H, \delta_s^vor(G)).$$
    For $n = 0$ we define $$\delta_s^v or(G) = e_0\wedge or(G).$$ For $n= 1$ let $$\delta_s^vor(G) = v_1\wedge (v_1,e_0)\wedge (v_2,e_0)\wedge v_2 \wedge d_vor(G).$$
\end{remark} $ $\\
One can check that $\delta_m, \delta_s,$ and $\delta_s+\delta_m$ are differentials, i.e. square to zero. As a vector space $MG_n$ is splits into a product of subspaces \begin{align*}MG_n = \prod_{k=0}^\infty \frac{C_kMG_n}{C_{k+1}MG_n}\end{align*} where $C_kMG_n$ is the subspace spanned by series of graphs with marking $M$ such that the first homology $H_1(M;\mathbb Q)$ has dimension at $k$. In this sense, $FG_n = {\frac{C_1MG_n}{C_1MG_n}}$. The spaces $\frac{C_kMG_n}{C_{k+1}MG_n} $ are complexes by themselves and inherit the differential $\delta_s+\delta_m$. \\
Finally, we can define $HG_n$ and $FG_n$. Namely, $(HG_n,\delta_s) =(HG_n,\delta_s+\delta_m)$ is the subcomplex of $(MG_n,\delta_s+\delta_m)$ spanned by series of graphs where all edges that are not honest hairs are marked. Let $CMG_n$ be the subspace spanned series of graphs $(G,M)$ with $H_1(M;\mathbb Q)$ being at least one-dimensional. Note that $\delta_s$ and $\delta_m$ leave $CMG_n$ invariant. Hence, $$FG_n  = \faktor{MG_n}{CMG_n}$$ is well-defined as a complex.
Although technically elements in $FG_n$ are only defined up to addition with elements in $CMG_n$, we still refer to them as graphs. More precisely, they are graphs where the marked edges form a forest.
\begin{comment}\begin{remark}

Note that the standard forested and hairy graph complexes $FG_n$ and $HG_n$ defined in literature are subcomplexes $FG_n(\bullet,0)$ and $HG_n(\bullet,0)$ spanned by series of graphs with an arbitrary number of honest hairs and without dishonest hairs. The homologies of these $FG_n(\bullet,0)$ compute the homology of a family of groups that generalize outer automorphism groups and automorphism groups of the free groups \cite{BrunWillwacher,ContantKassabovVogtmann,ContantVogtmann}. In \cite{BrunWillwacher} the homologies are computed for a low number of marked edges, hairs, and loop order. The homology of $HG_0(\bullet,0)$ is related to the homology of moduli space curves \cite{ChanGalatiusSørenPayne}.

\end{remark}
\end{comment}

We are interested in the homologies of $FG_n$ and $HG_n$ which are highly non-trivial. In contrast, the homology of $MG_n$ is zero. \begin{prop}\label{Thm MG}
    We have $$H(MG_n,\delta_s+\delta_m) = H(MG_n,\delta_m) = 0.$$
\end{prop}
\begin{proof}
First, we show that $H(MG_n,\delta_s+\delta_m) = H(MG_n,\delta_m)$. We have $$\delta_m\Delta\Delta+\Delta\Delta\delta_m = 2 \Delta \delta_m\Delta$$ that is $$[\delta_s,\Delta] = [[\delta_m,\Delta],\Delta] = 0.$$
Thus, we have that $$\exp(-\Delta) \delta_m \exp(\Delta) = \delta_m+\delta_s.$$
We can conclude that $\exp(\Delta)$ is an isomorphism of complexes between $(MG_n,\delta_s +\delta_m)$ and $(MG_n,\delta_m).$
It is left to show that $H(MG_n,\delta_m) = 0.$
Let $h:MG_n\to MG_n[-1]$ be the map that unmarks edges. That is, for admissible graphs $((G,M),or(G))$ we define $$h((G,M),or(G)) = \frac{1}{|E^{int}|}\sum_{e\in M\cap E^{int}}(G,M-{e},h^e(or(G)))$$
where for $n=0$ we let $$h^e(or(G)) = d_eor(G)$$ and for $n=1$ we let $$h^e(or(G)) = e\wedge or(G).$$
Then we have that $$h\delta_m-\delta_m h =  id.$$
Thus, the homology $H(MG_n,\delta_m) = 0$.
\end{proof}
The following result relates the homologies of the above-defined spaces $\frac{C_kMG_n}{ C_{k+1}MG_n}$. 
\begin{prop}\label{Thm CMG}
    For each $n= 0,1$ there is a spectral sequence with $$E^k_1 = H(\frac{C_kMG_n}{ C_{k+1}MG_n},\delta_s+\delta_m) $$ that converges to 0. Here $E^k_1$ denotes the $k$-th row of the first page of the spectral sequence. In particular, we have $$E^0_1 = H(FG_n,\delta_s+\delta_m).$$
\end{prop}
 \begin{proof}
For any graph $G$ with marking $M$ we let $\omega(G,M)$ denote the dimension of the homology of the marked edges dim$H_1(M;\mathbb Q)$. Then $C_kMG_n$ is spanned by series of graphs $(G,M)$ with $\omega(G,M)\geq k$. This induces a filtration of $MG_n$. We can construct a spectral sequence associated with this filtration with  $$E^k_1 = H(\frac{C_kMG_n}{ C_{k+1}MG_n},\delta_s+\delta_h).$$
 We have that deg$(G,M) \geq \omega(G,M)$ and can apply Proposition \ref{prop convergence}. Thus, the sequence converges to the total homology $H(MG_n,\delta_s+\delta_m) = 0$.
\end{proof}
\section{Additional Differentials}\label{section differentials}
In this section, we give a recipe for constructing differentials $D$ on $MG_n$ such that $\delta_s+\delta_m+ D$ is a differential. To achieve this, we pass to a larger complex $MG_n\subset MG_{n,\geq 2}$ of series of graphs where vertices with valence two are allowed. This complex is constructed like $MG_n$ but where any graph is admissible. On ${MG}_{n,\geq 2}$ a map that extends splitting a vertex on $MG_n$ can be defined in different ways. For now, it is defined as the author found it most natural, but in section \ref{Section Quasihair} it is defined slightly differently and the same following argument works. To this end, let $\Delta$ denote the map that splits vertices into two vertices of valence $\geq 3$ and let $\Delta_{\geq 2}$ denote the map that splits vertices into two vertices of valence $\geq 2$. Note that there is an inclusion $\iota: MG_n \hookrightarrow {MG}_{n,\geq 2}$ with $\Delta \iota = \iota\Delta$. Furthermore, marking edges in $MG_{n,\geq 2}$ is done as on $MG_n$ and yields a differential $\delta_m$. Naturally, we define $$\delta_s = [\delta_m,\Delta]$$ and $$\delta_{s,\geq 2} = [\delta_m,\Delta_{\geq 2}].$$ 
A subgraph of a graph $G$ is called a \emph{unmarked chain} if it is a sequence of unmarked edges connected by vertices of valence two in $G$. For any unmarked edge $e$ that is not a tadpole, we assign the \emph{chain length} $c(e)$ given by the number of the unique maximal unmarked chain it is part of. For tadpoles we set $c(e) = 2$. On ${MG}_{n,\geq 2}$ we define a map $$\Lambda: MG_{n,\geq 2} \to MG_{n,\geq 2}$$ that replaces maximal unmarked chains of length $k$ by unmarked chains of length $k+1$ (Figure \ref{fig:Lambda}). For graphs $(G,M,or(G))$ in $MG_{n\geq 2}$ with $G = (V,E)$ we define $$\Lambda(G,M,or(G)) = \sum_{e\in E-M} \frac{1}{c(e)}((V\cup \{v\},E-\{e\}\cup\{e_1,e_2\}),M,\Lambda^eor(G))$$ where $v$ is a new vertex, $e = \{v_1,v_2\},e_1 =\{v,v_1\},$ and $e_2 = \{v,v_2\}$. For $n=0$ the orientation remains unchanged $$\Lambda^eor(G) = or(G).$$ For $n=1$ is more complicated since unmarked edges are oriented we define $$\Lambda^eor(G) = (e_1,v_1)\wedge e_1\wedge (e_1,v)\wedge (e_2,v)\wedge e_2\wedge (e_2,v_2)\wedge d_{(e,v_1)}d_{e}d_{(e,v_2)} or(G).$$ Note that this expression is invariant of relabeling the subscript $1
$ and $2$. Naturally, the map $\Lambda$ induces a map $MG_{n\geq 2} \to MG_{n\geq 2} $. One can convince themselves that inserting vertices at different edges in a chain yields orientation-isomorphic graphs. Note that on $MG_{n\geq 2} $ the map $\exp(\Lambda)$ is defined. We can now formulate the recipe of constructing a differential $D$ of $MG_{n,\geq 2}$ with ingredient being a differential $\delta$ on $MG_{n\geq 2} $ that anticommutes with $\delta_{s,\geq 2}$.\\
\begin{figure}[h]
    \centering
    \includegraphics[scale = 0.6]{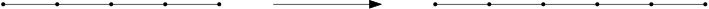}
    \caption{The action of $\Lambda$}
    \label{fig:Lambda}
\end{figure}
\begin{prop}\label{Prop Diff}
    Suppose that $\delta$ is a differential on $MG_{n,\geq 2}$ such that $\delta_{s\geq 2}\delta = -\delta\delta_{s,\geq 2}$. Then for the differential $D = \exp(\Lambda)\delta\exp(-\Lambda)$ we have $$(\delta_s+\delta_m+D)^2 =0.$$
\end{prop}
\begin{proof}
The proof essentially consists of noticing that $$\delta_{s,\geq 2} = \delta_s+ [\delta_m,\Lambda]$$
and that $$[\Lambda,[\delta_m,\Lambda]] = 0$$
so that $$\exp(-\Lambda)(\delta_m+\delta_s)\exp(\Lambda) = \delta_{s,\geq 2}.$$
We have that \begin{dmath*}
    (\delta_s+\delta_m)D = (\delta_s+\delta_m)(\exp(\Lambda)\delta\exp(-\Lambda)) = 
    \exp(\Lambda)\exp(-\Lambda) (\delta_s+\delta_m)\exp(\Lambda)\delta\exp(-\Lambda) =
    \exp(\Lambda)\delta_{s,\geq 2}\delta\exp(-\Lambda) = 
    - \exp(\Lambda)\delta\delta_{s,\geq 2}\exp(-\Lambda) = 
     - \exp(\Lambda)\delta\exp(-\Lambda)\exp(\Lambda)\delta_{s,\geq 2}\exp(-\Lambda) =  -D(\delta_s+\delta_m).
\end{dmath*}
Since $D^2 = (\delta_s+\delta_m)^2 = 0$ we have that $(\delta_s+\delta_m+D)^2 = 0$.
\end{proof}
Let $\pi:MG_{n\geq 2} \to MG_n$ be the natural projection. To find a differential on $MG_n$ we still need to check that $(\pi\circ D)^2 = 0$. In this case, we have that $$\delta_s+\delta_m+\pi\circ D$$
is a differential on $MG_n$ since $\pi$ and $\delta_s+\delta_m$ commute and thus\begin{dmath*}
    (\delta_s+\delta_m+\pi\circ D)^2 = (\delta_s+\delta_m)\circ \pi \circ D + \pi\circ D \circ (\delta_s+\delta_m) = 
    \pi\circ[(\delta_s+\delta_m)\circ D+D\circ (\delta_s+\delta_m)] = 0.
\end{dmath*}
In practice, it is far easier to check that a differential $\delta$ on $MG_{n\geq 2} $ anticommutes with $\delta_{s,\geq 2}$ than it is to check if it anticommutes with  $\delta_s+\delta_m$. In the rest of this section, we explore some examples and applications of Proposition \ref{Prop Diff}. Note that  $$H(MG_{n,\geq 2},\delta_s+\delta_m+D) = H(MG_{n,\geq 2} ,\delta_{s,\geq 2} + \delta)$$ since $\exp(\Lambda)$ establishes an isomorphism. On $HG_{n,\geq 2}$ we have $\delta_m = 0,D = \delta,\delta_s=\delta_{s,\geq 2},$ and $$H(HG_{n,\geq 2},\delta_s+\delta_m+D) = H(HG_{n,\geq 2} ,\delta_s + \delta).$$ In particular must $(\delta_s+\delta)^2 = 0$. If $D$ induces a differential on $FG_{n,\geq 2}$ then we also have $$H(FG_{n,\geq 2},\delta_s+\delta_m+D) = H(FG_{n,\geq 2} ,\delta_{s,\geq 2} + \delta).$$
In subsections \ref{Section Dishonest Hair},\ref{Section Connecting}, and \ref{Section Quasihair} we give examples of how to apply Proposition \ref{Prop Diff}. We then concentrate on the differential from section \ref{Section Dishonest Hair} for the rest of the paper.

\subsection{Attaching Dishonest Hairs}\label{Section Dishonest Hair}
The first differential we define is given by attaching dishonest hairs vertices and unmarked edges. We will apply Proposition \ref{Prop Diff} to a differential $\delta_h$ on $MG_{n,\geq 2}$ that attaches a dishonest hair to an internal vertex. For graphs $(G,M,or(G))$ in $MG_{n,\geq 2}$ with $G = (V,E)$ we define $$\delta_{h}(G,M,or(G)) = \sum_{v \in V^{int}}(V\cup\{v_0\},E\cup\{e_0\}),M\cup\{e_0\},\delta_h^vor(G))$$
where $e_0 = \{v,v_0\}$ and the orientation is given by $$\delta_h^vor(G)) = e_0\wedge or(G)$$
for $n=0$ and by $$\delta_h^vor(G)) = (v,e_0)\wedge (v_0,e_0)\wedge v_0\wedge or(G).$$
This induces a differential $\delta_h$ on $MG_{n,\geq 2}$ and it is easily checked that $\delta_h\delta_{s,\geq 2} = -\delta_{s,\geq 2}\delta_h$. Hence, after Proposition \ref{Prop Diff} we have that the map $D = \exp(\Lambda)\delta_h\exp(-\Lambda)$ is a differential. One can see that $$\delta_{h,e} = -[\delta_h,\Lambda]$$ is a differential that attaches a dishonest hair to an unmarked edge (Fig. \ref{fig: diff h,e}) and that $[\Lambda,\delta_{h,e}] = 0$. That is $$D = \delta_h+\delta_{h,e}.$$
Note that $D(MG_n) \subset MG_n$, thus $\delta_h+\delta_{h,e}$ defines a differential on $MG_n$ and in particular on $HG_n$. Clearly, $\delta_h+\delta_{h,e}$ leave $C_1MG_n$ invariant, therefore the differential induces a differential on $FG_n$. Hence, we found a differential $\delta_s+\delta_m+\delta_h+\delta_{h,e}$ on $FG_n$ and $HG_n$.
\subsection{Adding Interior Edges}\label{Section Connecting}
The following generalizes differentials on $HG_n(0,0)$ introduced in \cite{DGC} by Willwacher, Khoroshkin, and Živković.
For $n=0$ we define a differential $\delta_{v,v}$ that connects two vertices by a marked edge. Namely, for graphs $(G,M,or(G))$ with $G = (V,E)$ define $$\delta_{v,v}(G,M,or(G)) = \sum_{u,w\in V^{int}}((V,E\cup\{e_{u,w}\}),M\cup\{e_{u,w}\},e_{u,w}\wedge or(G)$$
where $e_{u,w} = \{u,w\}$. In $MG_{0,\geq 2}$ it is easily checked that $\delta_{s,\geq 2}\delta_c = -\delta_c\delta_{s,\geq 2}$. Thus Proposition \ref{Prop Diff} applies and $D= \exp(\Lambda)\delta_c\exp(-\Lambda)$ is so that $\delta_s+\delta_m+D$ is a differential. Let $$\delta_{v,e} = -[\delta_{v,v},\Lambda]$$ and $$\delta_{e,e} = -\frac{1}{ 2}[\delta_{v,e},\Lambda].$$ One can check that $$D = \delta_{v,v} +\delta_{v,e} +\delta_{e,e}.$$ The maps $\delta_{v,e}$ and $\delta_{e,e}$ are differentials themselves where $\delta_{v,e}$ connects unmarked edges and internal vertices and $\delta_{e,e}$ connects two unmarked edges (not necessarily distinct) by a marked edge. The map $D$ leaves $MG_0$ invariant, therefore $\delta_s+\delta_m+D$ defines a differential on $MG_0$. Moreover, $D$ defines a differential on $FG_0$ that anticommutes with $\delta_s+\delta_m$.
\begin{remark}
    In \cite{DGC} a further differential is defined on $HG_1(0,0)$. Although the differential is more complicated it certainly can be generalized to a differential on $MG_1$ by similar arguments. 
\end{remark}

\subsection{Attaching Quasihairs}\label{Section Quasihair}
We define a differential on $MG_0$ that attaches graphs to vertices that behave similarly to dishonest hairs. Namely, consider the graph $S$ (Figure \ref{fig:quasihair}) that consists of a marked edge $e_S$ with a tadpole attached to one end. For a graph $G$ and vertex $v$ let $G\cup_vS$ denote the graph obtained by attaching $S$ at the unique univalent vertex to $v$. We say a vertex is a \emph{quasileaf} if it is of valence three and has a tadpole attached. The edge creating the tadpole is called a \emph{flower}. For graphs $(G,M,or(G))$ in $MG_0$ we define $$\delta_q(G,M,or(G)) = \sum_{\substack{v\in V^{int}\\\text{no quasileaf}}}(G\cup_vS,M\cup\{e_S\},e_S\wedge or(G)).$$
This behaves very much like $\delta_h$ in section \ref{Section Dishonest Hair}. Although Proposition \ref{Prop Diff} does not directly apply since $\delta_{s\geq 2}$ acts on $S$ and $\delta_q$ and $\delta_{s,\geq 2}$ do not anitcommute. However, we can slightly alter the setup and make a similar argument work.  We alter $\Delta_{\geq 2}$ to $\tilde \Delta_{\geq 2}$ by subtracting the graphs where a flower is destroyed. Similarly, we change $\Lambda$ to $\tilde \Lambda$ by subtracting the graphs where a flower is destroyed. We then have that $$\exp(-\tilde \Lambda)(\delta_m+\delta_s)\exp(\tilde \Lambda) = \tilde \delta_{s,\geq 2} = [\delta_m,\tilde \Delta_{\geq2}]$$ and $$
\delta_q\tilde\delta_{s,\geq 2} = -\tilde\delta_{s,\geq 2}\delta_q.$$ The same proof as for Proposition \ref{Prop Diff} shows that $D = \exp(\tilde\Lambda)\delta_q\exp(-\tilde\Lambda)$ is a differential that commutes with $\delta_s+\delta_m$. It is given by $$D = \delta_q - [\delta_q,\tilde\Lambda] =\delta_q + \delta_{q,e}$$
where $\delta_{q,e}$ attaches $S$ to an unmarked edge. The map $D$ on $MG_0$ descends to a differential on $FG_0$ and $HG_0$ that anticommute with $\delta_s+\delta_m$.
\begin{remark}
    It would be interesting if there is an analog quasihair in the case of $n=1$.
\end{remark}
\begin{figure}[h]
    \centering
    \includegraphics[scale = 0.25]{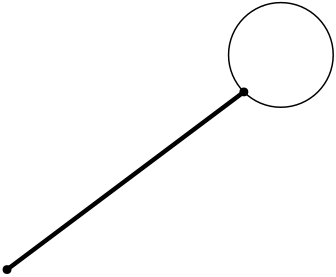}
    \caption{Graphs S: A quasihair.}
    \label{fig:quasihair}
\end{figure}

\section{Construction of the Spectral Sequences}\label{section spec seq}
We first construct the spectral sequences that we will use to prove Theorem \ref{section spec seq}.
\begin{prop}\label{prop spec seq}
    There is a spectral sequence with $$E_1 = H(FG_n,\delta_s+\delta_m)$$ and that converges to $H(FG_n,\delta_s+\delta_m+\delta_h+\delta_{h,e})$. Similarly, there is a spectral sequence that starts with $$E_1 = H(HG_n,\delta_s)$$ and that converges to $H(HG_n,\delta_s+\delta_h)$.
\end{prop}
\begin{proof}
Consider the subcomplexes $$(FG_n^{\geq k},\delta_s+\delta_m+\delta_h+\delta_{h,e}) \subset (FG_n,\delta_s+\delta_m+\delta_h+\delta_{h,e})$$
and 
$$(HG_n^{\geq k},\delta_s+\delta_h) \subset (HG_n,\delta_s+\delta_h)$$
that are spanned by series of graphs with at least $k$ hairs.
This defines descending filtrations $$FG_n = FG_n^{\geq 0} \supset FG_n^{\geq 1} \supset FG_n^{\geq 2} \supset ... $$
 and $$HG_n = HG_n ^{\geq 0} \supset HG_n^{\geq 1} \supset HG_n^{\geq 1} \supset ... \quad.$$
One can define spectral sequences associated with these filtrations that converge to the total homology (Proposition \ref{prop convergence}). The first pages are given by $H(FG_n,\delta_s+\delta_m)$ and $H(HG_n,\delta_s)$ respectively. 
\end{proof}
We now proceed to compute $H(FG_n,\delta_s+\delta_m+\delta_h+\delta_{h,e})$ and $H(HG_n,\delta_s+\delta_h)$. A class of graphs that play a major role in these homologies are graphs that we will call hedgehog graphs. We say a vertex of a graph $G$ in $MG_n$ which is not a leaf is \emph{critical} if it has valence three and a dishonest hair attached. Otherwise, a vertex which is not a leaf is called \emph{regular}. The set of \emph{hedgehog graphs} ins $MG_n$ is spanned by series of graphs having no regular vertices.  These are graphs with $k\geq 1$ interior edges forming a loop and an additional $k$ edges that connect the loop to $k$ hairs. The corresponding sets in $HG_n$ and $FG_n$ are denoted $HHG_n$ (hairy hedgehog graphs) and $FHG_n$ (forested hedgehog graphs) respectively (Figure \ref{fig: FGH}). \begin{figure}[h!]
    \centering
    \includegraphics[scale = 0.35]{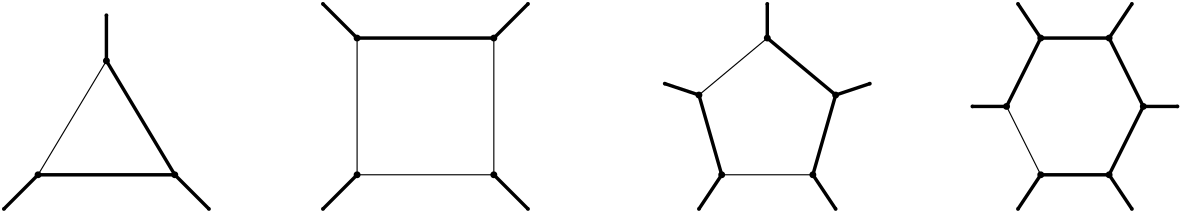}
    \caption{Forested hedgehog graphs}
    \label{fig: FGH}
\end{figure}

For convenience, we restate the definition of $H_n$. Namely, the graded vector space given by $$H_n = \prod_{\substack{ k\geq 1\\ k \equiv (-1)^n \text{ mod }4}}HH^k.$$
Let $RHG_n$ and $RFG_n$ denote the subcomplexes of $HG_n$ and $FG_n$ respectively spanned by series of graphs having at least one regular vertex. 
\begin{lemma}\label{lem regular}
    The homologies $$H(RHG_n,\delta_h) =0$$
    and $$H(RFG_n,\delta_h)=0$$ are trivial.
\end{lemma}
\begin{proof}
    For a graph $G$ let $\eta(G)$ denote the regular vertices.  Consider the following map $h:HG_n\to HG_n[-1]$. If $(G,M,or(G)) = ((V,E),M,or(G))$ is a graph with $\eta(G) > 0$ then let $$h(G,M,or(G)) = \frac{1}{ \eta(G)}\sum_{v\in HR}((V- \{v\},E- \{e\}),M-\{e\},h^vor(G))$$
where $e$ is the unique edge attached to $v$. For $n= 0$ the orientation is given by $$h^vor(G) = d_e or(G)$$ and for $n=1$ and $e=\{v,w\}$ where $v$ is the leaf, we define it as $$h^vor(G) = d_vd_{(v,e)}d_{(w,e)}or(G).$$
On $RHG_n$ we find $$h\delta_h+\delta_hh = id.$$
Therefore, $H(RHG_n,\delta_h) = 0$. An analogous map $\tilde h: RFG_n\to RFG_n[-1]$ can be defined to show that $H(FGH_n,\delta_h) = 0$.
\end{proof}
We proceed to calculate some more homologies.
\begin{prop}\label{prop first hom}
    We have $$H(HG_n,\delta_h) = H(HHG_n,0) = H_n[-1]$$ and 
    $$H(FG_n,\delta_m+\delta_h) = H(FHG_n,\delta_m) = H_n.$$
\end{prop}
\begin{proof}
We first compute $H(HG_n,\delta_h)$. For a graph $G$ let $\eta(G)$ denote the number of regular vertices. $HG_n$ splits as, $HHG_n\oplus RHG_n$. After Lemma \ref{lem regular} we have that $H(RHG_n,\delta_h) = 0$ and thus $H(HG_n,\delta_h) = H(HHG_n,\delta_h) = H(HHG_n,0) =HHG_n$.\\ It is only left to check that $HHG_n = H_n$. To see this, note that a graph $G$ with $\eta(G) = 0$ is trivial in $HHG_n$ if and only if there is a self-isomorphism $G\to G$ that inverts the orientation. The self-isomorphisms of such a graph with $k$ dishonest hairs are given by the symmetry group of a regular $k-$gon. One can check if $n=0$ orientation inverting symmetries exist if and only if $k$ is even. And if $n=1$ such symmetries exist if and only if $k$ is odd. Thus, $HHG_n$ is only non-trivial in degree $2k\geq 2$ for $k \equiv n+1$ mod $2$ (or equivalently in degree $k$ if $k-1 \equiv  (-1)^n$ mod $4$), where it is one-dimensional.\\
Now let us compute $H(FG_n,\delta_m+\delta_h)$. The forested graph splits into $FHG_n\oplus RFG_n$. Analogously to the case of $HG_n$ we find that $H(RHG_n,\delta_h) = 0$. Consider the filtration of $RHG_n$ that is given by the subcomplexes $RHG_n^{\geq k}$ spanned by series of graphs with at least $k$ marked internal edges. We can construct a spectral sequence associated to this filtration with $E_1 = H(RFG_n,\delta_h) = 0$  and that converges to $H(RFG_n,\delta_m+\delta_h)$. Thus, $H(RFG_n,\delta_m+\delta_h) = 0$ and $$H(FG_n,\delta_m+\delta_h) = H(FHG_n,\delta_m+\delta_h) = H(FHG_n,\delta_m).$$ Now consider $q:FHG_n\to FHG_n[-1]$ given by $$q((G,M,or(G)) = \frac{1}{ |E^{int}|}\sum_{e \in M -E^{int}}(G,M-\{e\},q^eor(G))$$
where for $n=0$ we define $$q^eor(G) = d_eor(G)$$ and for $n=1$ we let $$q^eor(G) = e\wedge or(G).$$
If $G$ is a forested hedgehog graph with $|E|-|M| > 1$ then $$(q\delta_m +\delta_mq)G = G.$$
We can conclude that any homology class is represented by graphs with $|E|-|M| = 1$. For $n=0$ note that graphs with $|HR|$ odd and $|E|-|M| = 2$ are trivial since there is an orientation inverting reflection. Moreover, graphs with $|E|-|M| = 1$ are trivial if and only if $|HR|$ is even. Conclude that the homology classes represented by graphs with $|E|-|M| = 1$ are nontrivial if and only if $|HR|$ is odd.\\
An analogous discussion can be done in the case $n=1$ to conclude that the nontrivial homology classes are given by graphs with $|HR|$ even and $|E|-|M| = 1$.\\
Forested hedgehog graphs with $|HR|= k$ and $|E|-|M| = 1$ have degree $2k-1$. Hence, $H(FHG_n,\delta_m)$ is one-dimensional in degree $2k-1\geq 1$ if $k\equiv n+1$ mod $2$ (or equivalently in degree $k\geq 1$ if $k \equiv (-1)^n$ mod 4). Conclude that $H(FG_n,\delta_m+\delta_h) = H(FHG_n,\delta_m) = H_n$.

\end{proof}
\begin{figure}
    \centering
    \includegraphics[scale = 0.3]{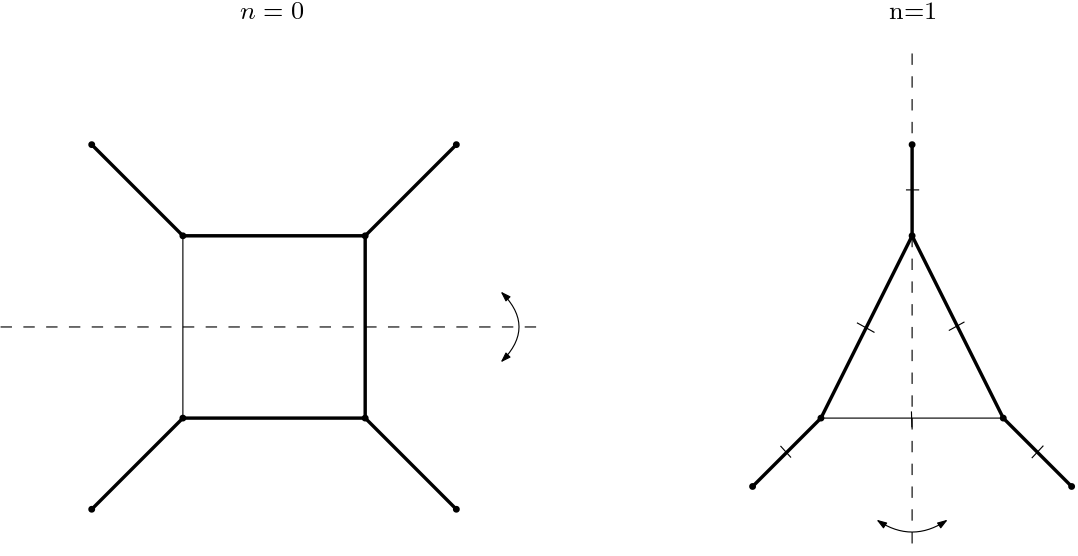}
    \caption{Symmetries of maximally marked forested hedgehog graphs.}
    \label{fig:sym FGH}
\end{figure}
These computations help us to compute the following.
\begin{prop}\label{prop second hom}
    We have $$H(HG_n,\delta_s+\delta_h) = H(HHG_n,0) = H_n[-1]$$ and $$H(FG_n,\delta_s+\delta_m+\delta_h+\delta_{h,e}) = H(FHG_n,\delta_m+\delta_{h,e}) = H_n.$$
\end{prop}
\begin{proof}
Let us first calculate $H(HG_n,\delta_s+\delta_h)$. This case is somewhat simpler than the case of forested hedgehog graphs. We define the subcomplexes $HG_n^{\geq k}$ spanned by series of graphs with at least $k$ interior edges. This induces a filtration $$HG_n = HG_n^{\geq 0} \supset HG_n^{\geq 1} \supset HG_n^{\geq 2} \supset ... \quad.$$ Since any edge has degree 1 we have a fixed degree $l$ and for large enough $k$ the degree $l$ part of $HG^\geq k$ is trivial. Thus, the associated spectral sequence converges to the total homology (Proposition \ref{prop convergence}). The first page of this spectral sequence is given by $$E_1 = H(HG_n,\delta_h) = H(HHG_n,0)$$
which coincides with the first page of the spectral sequence associated with the filtration of $(HHG_n,\delta_h+\delta_s)$ filtered by the number of interior edges. The latter spectral sequence converges to $H(HHG_n,\delta_s+\delta_h)$. Thus, we can conclude by standard spectral sequence result that the homologies coincide, i.e. $$H(HG_n,\delta_s+\delta_h) = H(HHG_n,\delta_s+\delta_h) = HHG_n = H_n.$$
With forested hedgehog graphs, one should be more careful, since filtering by the number of interior edges does not yield a bounded filtration. However, we have that $FG_n = FHG_n\oplus RFG_n$. In $FHG_n$ the differentials $\delta_s$, $\delta_h$, and $\delta_{h,e}$ are trivial and thus $$H(FHG_n,\delta_s+\delta_m+\delta_h+\delta_{h,e}) = H(FHG_n,\delta_m) = H_n.$$
Let $\zeta(G)$ denote the number of critical vertices of a graph and $\omega(G) = |E^{int}\cap M| +\zeta(G)$. We filter the complex $RFG_n$ by the value of $\omega$. That is, we define $RFG_n^{\geq k}$ be the subcomplex spanned by series of graphs with $\omega(G) \geq k$ and consider the filtration $$RFG_n = RFG_n^{\geq 0}\supset RFG_n^{\geq 1}\supset RFG_n^{\geq 2} \supset .... \quad.$$ Note that deg$(G) \geq \omega(G)$ and therefore Proposition \ref{prop convergence} applies. The associated spectral sequence converges to the homology $H(RHG_n,\delta_s+\delta_m+\delta_h+\delta_{h,e})$ and has $E_1 = H(RHG_n,\delta_h) = 0$ as the first page. Thus, $H(RHG_n,\delta_s+\delta_m+\delta_h+\delta_{h,e}) = 0$.
\end{proof}
Corollaries of Proposition \ref{prop spec seq} and Proposition \ref{prop second hom} are Theorem \ref{Thm FG} and Theorem \ref{Thm HG}.

\section*{Appendix A} \label{Appendix A}
\addcontentsline{toc}{section}{\protect\numberline{}Appendix}
In literature, the complexes $HG_n$ and $FG_n$ are often defined without dishonest hairs. In our notation, these complexes are described by $HG_n(\bullet,0)$ and $FG_n(\bullet,0)$. The homologies of $(FG_n,\delta_s+\delta_m)$ are directly related to the homology of a family of groups $\Gamma_{g,r}$ (see \cite{HatcherVogtmann}) that generalize the outer automorphism group of free groups $\Gamma_{g,0} = Out(F_g)$ and the automorphism groups of free groups $\Gamma_{g,1} = Aut(F_g)$. \begin{prop}
    Let $g,r\in \mathbb Z_{\geq 0}$ with $2g+r\geq 3$. Then \begin{enumerate}
        \item $H^k(FG_0^{g \text{ loop}}(r,0)) = H^k(\Gamma_{g,r};\mathbb Q)$
        \item $H^k(FG_1^{g \text{ loop}}(r,0)) = H^k(\Gamma_{g,r};sgn)$
    \end{enumerate}
    where $\mathbb Q$ is the trivial representation of $\Gamma_{g,r}$ and $sgn$ is the sign representation.\qed
\end{prop}
This result can be found in \cite[Theorem 2.5]{BrunWillwacher}. The results in (i) are proven by Contant, Kassabov, and Vogtmann for $r=0$ in \cite{ContantKassabovVogtmann} and by J. Contant and K. Vogtmann for general $r$ in \cite{ContantVogtmann}. In \cite{BrunWillwacher} the homologies are computed for a low number of marked edges, hairs, and loop order.

 In the following we construct a surjective projections $$H(HG_n(r,k)) \to H(HG_n(r-l,k+l)[-l]) $$ and $$H(HG_n(r,k)) \to H(HG_n(r-l,k+l)[-l])$$
 for $l\leq r$. Firstly, there is a natural projection $$\pi: MG_n(r,k) \to MG_n(r-l,k+l)[-l]$$
that maps any graph $G$ to the corresponding graph where the last $l$ honest hairs are marked and the orientation is given by $$or(\pi G) = h_{r-l+1}\wedge ...\wedge h_r\wedge or(G)$$
for $n=0$ and by $$or(\pi G) = d_{h_{r-l+1}}...d_{h_r} or(G)$$
for $n=1$,
where $h_1,...,h_r$ is the ordering of the hairs of $G$. This projection commutes with the differentials and induces projections on hairy graphs and forested graphs. For $l=r 1$ and $k=0$ the map $\pi$ is bijective, in particular for $$H(HG_n(1,0),\delta_s) = H(HG_n(0,1),\delta_s)[-1])$$
and $$H(FG_n(1,0),\delta_s+\delta_m) = H(FG_n(0,1),\delta_s+\delta_m)[-1]).$$
Moreover, for any $r\in \mathbb Z_{\geq 0}$ the projection $\pi$ induces a subjective projection in homology. We can also define a map in the other direction $$\iota:MG_n(r,k)\to MG_n(r+l,k-l)[l].$$
For an element $(G,M,or(G))$ in $MG_n$ pick any graph representing the equivalence class. We unmark $l$ dishonest hairs and add them at the end of the ordering of the honest hairs. If we unmark $l$ of $k$ dishonest hairs, then there are $\frac{r!}{(r-l)!}$ ways to enumerate $l$ dishonest hairs. Let $h = (h_{r+1},...,h_{r+l})$ denote a possible enumeration. Then for $n=0$  we define $$\iota (G,M,or(G)) = \frac{(r-l)!}{ r!}\sum_{h} (G,M-\{h_{r+1},...,h_{r+l}\},d_{h_{r+l}}...d_{h_{r+1}}or(G))$$
and for $n=1$ we define $$\iota (G,M,or(G)) = \frac{(r-l)!}{ r!}\sum_{h} (G,M-\{h_{r+1},...,h_{r+l}\},h_{r+l}\wedge ...\wedge h_{r+1}\wedge or(G)).$$
The map $\iota$ is a chain map and $\pi\circ\iota = \text{id}_{MG_n}$.

\begin{prop}\label{Prop AutoGroups}
    For $l\in \mathbb Z_{\geq 0}$ there is an inclusion $$H(HG_n(l),\delta_s) \hookrightarrow H(OHG_n(l)[l],\delta_s)$$
    and $$H(FG_n(l),\delta_s+\delta_m) \hookrightarrow H(OFG_n(l)[l],\delta_s+\delta_m).$$
    The induced map $\pi_*$ is the projection to these subspaces.\qed
\end{prop} 
More conceptually, there are two actions of the symmetric group $S_l$ on $H(HG_n(r,k))$ and $H(FG_n(r,k)$ respectively that reorder the last $l$ honest hairs, where one of the actions adds an additional sign of the permutation. The first action corresponds to forgetting the ordering of the last $l$ honest hairs, and the second action checks if there are orientation-inverting graph self-isomorphisms. Taking equivalence classes induced by both actions gives the projection $\pi_*$.
\section*{Appendix B}
The following tries to imitate Theorem \ref{Thm FG} and Theorem \ref{Thm HG} but is less precise in the sense that the limit of the spectral sequence could still be complicated. Note that $\delta_{v,v}+\delta_{v,e}+\delta_{e,e}$ is defined in section \ref{Section Connecting} and $\delta_{q}+\delta_{q,e}$ in section \ref{Section Quasihair}.
\begin{lemma}\label{lem spec seq}
    Let $C$ be the complex $HG_0(r,k)$ or $FG_n(r,k)$ for $r,k\in\mathbb Z_{\geq 0}$ and $D$ be the differential $\delta_{v,v}+\delta_{v,e}+\delta_{e,e}$ or $\delta_{q}+\delta_{q,e}$. There exists a spectral sequence with $$E^1 = H(C,\delta_s+\delta_m)$$
    that converges to $$H(C,\delta_s+\delta_m+D).$$
\end{lemma}
\begin{proof}
The subcomplexes of $C$ spanned by series of graphs with looporder $\geq k$ for some $k\in\mathbb Z_{\geq 1}$ define a complete filtration.  One can set up a spectral sequence associated with this filtration that converges to the homology $H(C,\delta_s+\delta_m+D)$ \cite[Complete Convergence Theorem 5.5.10]{Weibel}. The first page of this spectral sequence is given by $$E^1 = H(C,\delta_s+\delta_m).$$
\end{proof}

Although the homologies $H(C,\delta_s+\delta_m+D)$ are probably still complicated, it might be fruitful to investigate them.
\section*{Appendix C}
We will very briefly discuss the results we used for the convergence of the spectral sequence. All the constructions and convergence results can be found in \cite[5 Spectral Sequence]{Weibel}. For a filtration $F^kC$ of a complex, $C = (C^\bullet,\delta)$ the 0-th page of the spectral sequences is given by $$E_0 ^{n,k} = \faktor{F^kC^n}{F^{k+1}C^n}\quad.$$
\begin{prop}\label{prop convergence}
    If for any $k$ all but finitely many groups $E_0^{n,k}$ are trivial, then the spectral sequence converges to the homology $H(C,\delta)$.
\end{prop}
\begin{proof}
We have that the filtration is bounded. Then $\cite[\text{Classical Convergence 5.5.1}]{Weibel}$ implies the result.
\end{proof}
\section*{Acknowledgement}
\addcontentsline{toc}{section}{\protect\numberline{}Acknowledgement}
This paper evolved from a semester thesis I wrote at ETH Zürich under the supervision of Thomas Wilwacher. I thank him for helpful, insightful, and confusion-resolving discussions.

\end{document}